\documentclass[12pt, reqno]{amsart}
\usepackage{amsmath, amsthm, amscd, amsfonts, amssymb, graphicx, color}
\usepackage[bookmarksnumbered, colorlinks, plainpages]{hyperref}

\newtheorem{theorem}{Theorem}[section]
\newtheorem{lemma}[theorem]{Lemma}
\newtheorem{proposition}[theorem]{Proposition}
\newtheorem{corollary}[theorem]{Corollary}
\theoremstyle{definition}
\newtheorem{definition}[theorem]{Definition}
\newtheorem{example}[theorem]{Example}

\theoremstyle{remark}
\newtheorem{remark}[theorem]{Remark}
\numberwithin{equation}{section}

\begin{document}
\setcounter{page}{1}

\title[Convex Optimization in Hilbert spaces]{On Certain Conditions for  Convex Optimization in Hilbert Spaces}

\author[N. B. Okelo]{N. B. Okelo}

\address{Department of Pure and Applied Mathematics\\ School of Mathematics and Actuarial Science\\ Jaramogi Oginga Odinga  University of Science and Technology\\ Box 210-40601, Bondo-Kenya.}
\email{bnyaare@yahoo.com}



\date{Received: xxxxxx; Revised: yyyyyy; Accepted: zzzzzz.
\newline \indent $^{*}$ Corresponding author}

\begin{abstract}
In this paper convex optimization techniques are employed for
 convex optimization problems in infinite dimensional Hilbert spaces. A first order optimality condition is given.
Let  $f : \mathbb{R}^{n}\rightarrow \mathbb{R}$ and let $x\in \mathbb{R}^{n}$ be a
local solution to the problem $\min_{x\in \mathbb{R}^{n}} f(x).$ Then
$f'(x,d)\geq 0$ for every direction $d\in \mathbb{R}^{n}$  for which $f'(x,d)$ exists. Moreover, Let  $f : \mathbb{R}^{n}\rightarrow \mathbb{R}$ be differentiable at  $x^{*}\in \mathbb{R}^{n}.$ If $x^{*}$ is a local minimum
of $f$, then $\nabla f(x^{*}) = 0.$ A simple application involving the Dirichlet problem is also given.
\end{abstract} \maketitle

\section{Introduction}
Studies on optimization has attracted the attention of many mathematicians and researchers
 over along period of time(see \cite{Boy}, \cite{Ber}, \cite{Eke1}, \cite{Eke2}, \cite{Gra} and the references there in).
  In this paper, we are concerned  with the classical results
  on optimization of convex functionals
in infinite-dimensional real Hilbert spaces. When working with infinite-dimensional
spaces, a basic difficulty is that, unlike the case in finite-dimension, being closed and
bounded does not imply that a set is compact\cite{Alb}. In reflexive Banach spaces, this problem
is mitigated by working in weak topologies and using the result that the closed
unit ball is weakly compact\cite{Via}. This in turn enables mimicking some of the same ideas in
finite-dimensional spaces when working on unconstrained optimization problems\cite{Dun}.
It is the goal of this paper to provide a concise coverage of the problem of minimization
of a convex function on a Hilbert space. The focus is on real Hilbert spaces, where
there is further structure that makes some of the arguments simpler. Namely, proving
that a closed and convex set is also weakly sequentially closed can be done with an elementary
argument, whereas to get the same result in a general Banach space we need
to invoke Mazurs Theorem\cite{Glo}. The ideas discussed in this brief note are of great utility
in theory of PDEs, where weak solutions of problems are sought in appropriate Sobolev
spaces\cite{Kur}. After a brief review of the requisite preliminaries, we develop the
main results. Though, the results in this note are
classical, we provide proofs of key theorems for a self contained presentation. A simple
application, regarding the Dirichlet problem, is provided  for the purposes
of illustration. Before moving further we recall an important point about notions of compactness
and sequential compactness in weak topologies. It is common knowledge that compactness
and sequential compactness are equivalent in metric spaces. The situation is not
obvious in the case of weak topology of an infinite-dimensional normed linear space\cite{Eke2}.

\section{Preliminaries}
\begin{definition}
A sequence $x_{n}$ in a Banach space $B$ is said to converge to $x\in B$ if $\lim_{n\rightarrow \infty}x_{n}=x.$  Also a sequence $x_{n}$ in a Hilbert space $H$ converges weakly to $x$ if,
$\lim_{n\rightarrow \infty} \langle x_{n}, u\rangle= \langle x, u\rangle, \;\forall u\in H.$
We use the notation $x_{n} \rightharpoonup x$ to mean that $x_{n}$ converges weakly to $x.$
\end{definition}

\begin{definition}
A  real valued function $f$ on a Banach space $B$ is lower semi-continuous
(LSC) if $f(x) \leq \lim\inf_{n\rightarrow \infty}f(x_{n})$
for all sequences $x_{n}$ in $B$ such that $ x_{n} \rightarrow x$ (strongly) and weakly sequentially lower semi-continuous (weakly sequentially LSC) if $ x_{n} \rightharpoonup x.$
\end{definition}

\begin{definition}
A non-empty set $W$ is said to be convex if for  all $ \beta \in [0, 1]$ and $\forall\; x, y \in W$
$\beta x + (1-\beta)y \in W.$
 Let $X$ be a metric space and $W\subseteq X$ a non-empty convex set. A function
$f :W\rightarrow \mathbb{R}$ is convex if for all $ \beta \in [0, 1]$ and $\forall\; x, y \in W$
$$f(\beta x + (1-\beta)y)\leq \beta f(x) + (1-\beta)f(y).$$
\end{definition}
\begin{remark}
 We note that the function $f$ in the above definition is called strictly convex if the above
inequality is strict for $x \neq y$ and $\beta \in (0, 1).$ A function $f$ is convex if and only if its epigraph, $epi(f)$, is convex whereby
$epi(f) := f(x, r) \in dom(f) \times \mathbb{R} : f(x) \leq r.$ An optimization problem is convex if both the objective function and feasible set are convex(see\cite{Eke2} for details).

\end{remark}

\begin{definition}  Let $\mathbb{R}^{n}$ be an $n$-dimensional real space and $W \subseteq \mathbb{R}^{n}$.
A point $x^{*}\in \mathbb{R}^{n}$ is called a \emph{global minimizer} of the optimization problem $\min_{ x\in W} f(x),$ if $x^{*}\in W$
 and $f(x^{*}) \leq f(x),$ for all $x \in W.$
\end{definition}

\begin{definition} Let $\mathbb{R}^{n}$ be an $n$-dimensional real space and $W \subseteq \mathbb{R}^{n}$.
A point $x^{*}\in \mathbb{R}^{n}$ is called a \emph{local minimizer} of the optimization problem $\min_{ x\in W} f(x),$ if there exists a neighbourhood $N$ of $x^{*}$ such that $x^{*}$ is a global minimizer of the problem $\min_{ x\in W\cap N} f(x).$ That is there exists $\varepsilon > 0 $ such that $f(x^{*}) \leq f(x),$ whenever $x^{*}\in W$ satisfies $\|x^{*}-x\|\leq \varepsilon .$

\end{definition}
\begin{remark}
Any local minimizer of a convex optimization problem is a
global minimizer\cite{Boy}.
\end{remark}

\begin{proposition}
Let $B$ be a Banach space and $f : B \rightarrow \mathcal{\mathbb{R}}.$ Then the following are conditions \cite{Ber}
equivalent.
(i). $f$ is (weakly sequentially) LSC.\\
(ii). $epi(f),$ is (weakly sequentially) closed.\\
\end{proposition}
\begin{remark}
$f : B \rightarrow \mathcal{\mathbb{R}}$ is coercive if for all $x\in B,$ $\lim_{\|x\|\rightarrow\infty}f(x)=\infty.$

\end{remark}

\begin{proposition}\label{prop2}Let $H$ be an infinite dimensional real separable Hilbert space and let $W \subseteq H$ be a (strongly) closed and convex set. Then, $W$ is weakly sequentially closed.
\end{proposition}
\begin{proof}
Let the sequence $ x_{n} \rightharpoonup x$ be in $W.$ It only suffices to show that $x\in W$ by showing that $x=\phi_{W}(x),$ where $\phi_{W}(x)$ is the projection of $x$ into the closed convex set $W$. Indeed, we know that the projection $\phi_{W}(x)$ satisfies the variational inequality,
$\langle x - \phi_{W}(x), y- \phi_{W}(x)\rangle \leq 0,$ for all $y \in W.$ \\
So,
\begin{equation}\label{eq1}
   \langle x - \phi_{W}(x), x_{n}- \phi_{W}(x)\rangle \leq 0,\; \forall \, n.
\end{equation}

But, $ x_{n} \rightharpoonup x$ be in $W$ so we have,
\begin{eqnarray*}
  \|x - \phi_{W}(x)\|^{2} &=& \langle x - \phi_{W}(x), x - \phi_{W}(x)\rangle \\
   &=& \lim_{n\rightarrow \infty} \langle x - \phi_{W}(x), x_{n} - \phi_{W}(x)\rangle
\end{eqnarray*}
Hence, by Equation \ref{eq1} we have $\|x - \phi_{W}(x)\|=0.$  That is, $x=\phi_{W}(x).$
\end{proof}

\begin{lemma}
Let $f : H \rightarrow \mathbb{R}$ be a LSC convex function. Then $f$ is weakly LSC.
\end{lemma}
\begin{proof} We know that  $f$ is convex iff $epi(f)$ is convex. Moreover, $epi(f)$
is strongly closed  because $f$ is (strongly) LSC. By Proposition \ref{prop2}  we have that $epi(f)$ is weakly
sequentially closed  implying that $f$ is weakly sequentially LSC.

\end{proof}

\section{Main Results}
\begin{theorem}
\label{thm1}Let $H$ be an infinite dimensional real separable Hilbert space and $W \subseteq H$ be a weakly sequentially closed and bounded set. Let $f : W \rightarrow \mathbb{R}$ be weakly sequentially LSC. Then $f$ is bounded from below and has a
minimizer on $W$.
\end{theorem}
\begin{proof}
The proof has two steps:\\ (i).  $f$ is bounded below. \\ (ii). There exists a minimizer in $W.$
\vskip 5mm
\emph{Step(i)}: Suppose  that $f$ is not bounded from below. Then there exist a sequence $x_{n} \in W$ such that $f(x_{n}) < -n$ for all
$n.$ But  $W$ is bounded so $x_{n}$ has a weakly convergent subsequence $x_{n_{i}}$
Furthermore, $W$ is weakly sequentially closed therefore $x\in W$. Then, since f is weakly
sequentially LSC we have $f(x) \leq \lim\inf_{n\rightarrow \infty}f(x_{n_{i}})=-\infty$ which is a contradiction. Hence,
$f$ is bounded from below.\\
\emph{Step(ii)}: Let  $x_{n} \in W$ be a minimizing sequence
for $f$ that is $f(x_{n}) \rightarrow \inf_{W}f(x).$ Let $\lambda := \inf_{W}f(x).$ Since $W$ is bounded and weakly
sequentially closed, it follows by \cite{Gra} that $x_{n}$ has a weakly convergent subsequence has a weakly convergent subsequence $x_{n_{i}}\in W$. But $f$ is weakly sequentially LSC  so we have
$$\lambda\leq f(x^{*}) \leq \lim \inf f(x_{n_{i}} ) = \lim f(x_{n_{i}} ) = \lambda$$
So, $f(x^{*}) = \lambda$
\end{proof}

\begin{corollary}
Let $H$ be an infinite dimensional real separable Hilbert space and $W \subseteq H$ be a weakly sequentially closed and bounded set. Let $f : W \rightarrow \mathbb{R}^{n}$ be non-empty and closed, and that $f : W \rightarrow \mathbb{R}^{n}$ is LSC and coercive.
Then the optimization problem $\inf_{x\in W}f(x)$ admits at least one global minimizer.
\end{corollary}

\begin{proof}
By \cite{Boy} with an analogy to the proof of Theorem \ref{thm1} the proof of coercivity is sufficient.
\end{proof}
\begin{theorem}
A function that is strictly convex on $W$ has a unique minimizer on W$.$

\end{theorem}

\begin{proof}
Assume the contrary, that f $(x)$ is convex yet there are two points $x,  y \in W$ such that $f (x)$ and $f (y)$
are local minima. Because of the convexity of $W$ every point on the secant line $\beta x + (1 - \beta)y$ is in $W.$ Without
loss of generality suppose $f (x) \geq f (y)$ if this is not the case, simply relabel the points. We then have
$\beta f (x) + (1 - \beta) f (y) < f (y),\forall\;  \beta\in (0, 1).$
But $f$ is  strictly convex, we also have $f (\beta x + (1 -\beta)y) < f (x),\forall\;  \beta\in (0, 1).$
Taking $\beta$ arbitrarily close to $0 $ along the secant line, $z = \beta x + (1 -\beta)y$ remains in $W $(since $W$ is convex)
and $f (z)$ remains strictly below $f (x)$ (because $f$ is strictly convex). Therefore, there is no open ball $B$  containing $x$ such
that $f (x) < f (z),\forall\; z (B \cap W)\setminus x.$ Therefore, $x$ is not a local minimizer, which is a contradiction.
\end{proof}

\noindent In this last part we give an optimality conditions. We give the first order condition for optimality here.
Consider the function $\psi : \mathbb{R} \rightarrow \mathbb{R}$ given by
$\psi(t) = f(x + td)$ for some choice of $x$ and $d$ in $ \mathbb{R}^{n}$. The key variational object in this context is the directional
derivative of $f$ at a point $x$ in the direction $d$ given by $$f'(x,d)=\lim_{t \downarrow 0}\frac{f(x + td) -f(x)}{t}.$$
When $f$ is differentiable at the point $ x\in\mathbb{R}^{n}$, then
$f'(x, d) = \nabla f(x)^{T} d = \psi'(0).$ The next two results give us  an optimality condition.

\begin{proposition}\label{PROP}
Let  $f : \mathbb{R}^{n}\rightarrow \mathbb{R}$ and let $x\in \mathbb{R}^{n}$ be a
local solution to the problem $\min_{x\in \mathbb{R}^{n}} f(x).$ Then
$f'(x,d)\geq 0$ for every direction $d\in \mathbb{R}^{n}$  for which $f'(x,d)$ exists.
\end{proposition}
\begin{theorem}\label{thmL}
Let  $f : \mathbb{R}^{n}\rightarrow \mathbb{R}$ be differentiable at  $x^{*}\in \mathbb{R}^{n}.$ If $x^{*}$ is a local minimum
of $f$, then $\nabla f(x^{*}) = 0.$
\end{theorem}
\begin{proof}
We know that every differentiable function is continuous so by Proposition \ref{PROP} we have we have $$0 \leq f'(x^{*}, d) = \nabla f(x^{*})^{T} d,$$ for all $d\in \mathbb{R}^{n}$.
Taking $d = -\nabla f(x^{*})$ we obtain
$0 \leq - \nabla f(x^{*})^{T} \nabla f(x^{*}) = - \|\nabla f(x^{*})\|^{2}\leq 0.$
Therefore,  $\nabla f(x^{*}) = 0.$
\end{proof}

\begin{example}
Consider the Dirichlet problem: $-\triangle u=f,$  in $W$ and $ u=0,$  on $\partial W,$ where $W\subset \mathbb{R}^{n}$ is a bounded domain, and $f \in  L^{2}(W).$ It is well known that this problem has a weak solution weak
which is convex and continuous, and coercive. Thus, the existence of a unique
minimizer is ensured by application of Theorem \ref{thmL}.
\end{example}
\section{Conclusion}
With regard to Portfolio Optimization, this study is geared towards applications to particularly Stochastic optimization with consideration  to: Cox-Ross-Rubinstein model and Hamilton-Jacobi-Bellman Equation\cite{Ber}.

{\bf Acknowledgement.} The author is thankful to NRF, Kenya for the financial support no NRF/JOOUST/2016/2017-001 towards this research.

\bibliographystyle{amsplain}

\begin{thebibliography}{99}
\bibitem{Alb} F. Albiac and N. J. Kalton, \textit{Topics in Banach space theory}, Volume 233 of Graduate
Texts in Mathematics. Springer, New York, 2006.

\bibitem{Boy} S. Boyd and L. Vandenberghe, \textit{Convex Optimization}, Cambridge University
Press, Cambridge, United Kingdom, 2004.

\bibitem{Ber} D. P. Bertsekas,   \textit{Convex Analysis and Optimization}, Athena Scienti.c, Belmont,
MA, 2003.

\bibitem{Dun} N. Dunford and J. T. Schwartz, \textit{Linear operators}, Part I. Wiley Classics Library. John
Wiley and Sons Inc., New York, 1988.

\bibitem{Eke1} I. Ekeland and R. Temam, \textit{Convex Analysis and Variational Problems},
North Holland, Amsterdam, 1976.

\bibitem{Eke2} I. Ekeland and T. Turnbull, \textit{Infinite Dimensional Optimization
and Convexity}, The University of Chicago Press, Chicago, 1983.

\bibitem{Glo} R. Glowinski, J. L. Lions and R. Tremolieres, \textit{ Numerical Analysis of
Variational Inequalities}, North Holland, Amsterdam, 1981.

\bibitem{Gra} M. Grasmair, \textit{Minimizers of optimization problems}, To appear.

\bibitem{Kur} A.J. Kurdila and M. Zabarankin., \textit{Convex functional analysis}, Systems and Control:
Foundations and Applications. Birkhauser Verlag, Basel, 2005.

\bibitem{Via} J.P. Vial, \textit{Strong  convexity of set and functions}, J. Math. Econom \textbf{9} (1982), no. 1-2, 187--205.


\end{thebibliography}

\end{document}